%% file: main.tex
\documentclass[11pt]{article}
\usepackage[utf8]{inputenc}
\usepackage[document]{ragged2e}
\usepackage{graphicx}
\usepackage{fancyhdr}
\usepackage{hyperref}
\usepackage{titlesec}
\usepackage{authblk}
\usepackage[backend=biber,maxbibnames = 99]{biblatex}
\usepackage{mathrsfs}
\usepackage{amsmath}   
\usepackage{amssymb}
\usepackage{enumitem}
\usepackage[margin=1.1in,a4paper]{geometry}
\usepackage[UKenglish]{datetime2}
\hypersetup{
    colorlinks,
    citecolor={[rgb]{0,0.2,0.6}},
    filecolor={[rgb]{0,0.2,0.6}},
    linkcolor={[rgb]{0,0.2,0.6}},
    urlcolor= {[rgb]{0,0.2,0.6}}
}

\def\d{\delta}
\def\g{\gamma}
\def\G{\Gamma}
\def\k{\kappa}
\def\l{\lambda}

\def\m{\mu}
\def\n{\nu}
\def\s{\sigma}

\def\p{\pi}

\def\x{\xi}

\def\z{\zeta}

\newcommand{\R}{{\mathbb R}}

\newcommand{\N}{\mathbb{N}}

\newcommand{\supp}{\text{supp}}
\newcommand{\Hull}{\text{Hull}}

\newcommand{\sP}{{\mathscr P}}

\newcommand{\sW}{{\mathscr W}}

\newcommand{\domain}{\mathbb{R}^d}
\newcommand{\co}{\colon}

\usepackage{amsthm}
\theoremstyle{plain}
\newtheorem{theorem}{Theorem}[section]
\newtheorem{defn}[theorem]{Definition}
\newtheorem{lemma}[theorem]{Lemma}

\newtheorem{example}[theorem]{Example}

\newtheorem{propn}[theorem]{Proposition}
\newtheorem{remark}[theorem]{Remark}

\titleformat{\section}[block]
  {\normalfont\scshape\filcenter}{\thesection}{1em}{}
\makeatletter
\renewcommand\subsection
  {%
    \@startsection
      {subsection}
      {2}
      {\z@}
      {3.25ex \@plus 1ex \@minus .2ex}
      {-1em}
      {\normalfont\normalsize\bfseries}%
  }
\renewcommand\subsubsection
  {%
    \@startsection
      {subsubsection}
      {3}
      {\z@}
      {3.25ex \@plus 1ex \@minus .2ex}
      {-1em}
      {\normalfont\normalsize\bfseries}%
  }
\makeatother

\setitemize{itemsep=5pt,topsep=5pt,parsep=0pt,partopsep=0pt}
\newenvironment{nospaceflalign*}
 {\setlength{\abovedisplayskip}{5pt}\setlength{\belowdisplayskip}{5pt}%
  \csname flalign*\endcsname}
 {\csname endflalign*\endcsname\ignorespacesafterend}

\renewenvironment{abstract}{%
\hfill\begin{minipage}{0.95\textwidth}
\rule{\textwidth}{1pt}}
{\par\noindent\rule{\textwidth}{1pt}\end{minipage}}

\makeatletter

\renewcommand\@maketitle{%
\hfill
\begin{minipage}{0.95\textwidth}
\begin{centering}
\vskip 2em
\let\footnote\thanks 
{\LARGE \@title \par }
\vskip 1.5em
{\Large \@author \par }
\end{centering}
\end{minipage}
\vskip 1em \par
}

\makeatother
\title{Some Convexity Criteria for Differentiable Functions on the $2$-Wasserstein Space}
\author{Guy Parker\thanks{Email: guy.m.parker@durham.ac.uk}}
\affil{\normalsize{\emph{Department of Mathematical Sciences, Durham University, United Kingdom}}}

\begin{document}
\maketitle
\justify
\input{chapters/abstract.tex}

\input{chapters/introduction.tex}
\input{chapters/notation.tex}
\input{chapters/preliminaries.tex}
\input{chapters/chapterone.tex}
\input{chapters/chaptertwo.tex}
\input{chapters/chapterthree.tex}
\input{chapters/acknowledgements.tex}
\printbibliography
\end{document}

%% file: chapters/abstract.tex
\thispagestyle{plain}
\begin{abstract}
\begin{center}
    Abstract
\end{center}
We show that a differentiable function on the $2$-Wasserstein space is geodesically convex if and only if it is also convex along a larger class of curves which we call `acceleration-free'. In particular, the set of acceleration-free curves includes all generalised geodesics. We also show that geodesic convexity can be characterised through first and second-order inequalities involving the Wasserstein gradient and the Wasserstein Hessian. Subsequently, such inequalities also characterise convexity along acceleration-free curves.
\end{abstract}

%% file: chapters/introduction.tex
\section{Introduction} 
\subsection{Acceleration-Free Curves and Main Result.}
In the theory developed in the book of Ambrosio, Gigli and Savar\'e \cite{Ambrosio2008}, the notion of \emph{geodesic semi-convexity} is a key ingredient in establishing the existence and uniqueness of gradient flows in the $2$-Wasserstein space $(\sP_2(\domain),W_2)$. 
However, without the stronger notion of \emph{semi-convexity along generalised geodesics}, we can not use the theory developed in \cite{Ambrosio2008} to establish several other important properties of gradient flows such as stability and optimal error estimates. 
Despite the strength of convexity along generalised geodesics over geodesic convexity in this regard, it is shown in \cite[Chapter 9]{Ambrosio2008} that, for the three energy functionals introduced by McCann in \cite{McCann1997}, these two notions coincide. 
This manuscript shows that this occurrence is not limited to these energy functionals. 
On the contrary, in section \hyperref[sec:one]{Section Three}, we prove the following Theorem.
\begin{theorem}\label{geodesicequivalencetheorem}
Let $F\co\sP_2(\domain)\to\R$ and let $\l \in \R$. 
If $F$ is differentiable on $\sP_2(\domain)$ then the following statements are equivalent.
 \begin{enumerate}[itemsep=2pt,topsep=3pt]
    \item $F$ is $\l$-geodesically convex.
    \item $F$ is $\l$-convex along generalised geodesics.
    \item $F$ is $\l$-convex along acceleration-free curves.
 \end{enumerate}
\end{theorem}
Theorem \ref{geodesicequivalencetheorem} introduces the notion of \emph{acceleration-free curves} (cf. Definition \ref{accelerationfree}) which are, heuristically, curves of probability measures which describe the evolution of a density of particles for which the path of each particle is acceleration-free.
For continuously differentiable functions on the $2$-Wasserstein space, the notion of convexity along acceleration-free curves coincides with the first-order notions of \emph{$L$-convexity} and \emph{displacement monotonicity}, introduced by Carmona and Delarue (cf. \cite[Definition 5.70]{CarmonaDelarue}) and Ahuja (cf. \cite[Equation 4.3]{Saran}) respectively. The latter notion of \emph{displacement monotonicity} has seen particular success in the study of mean field games (see Ahuja \cite{Saran}; Gangbo et al. \cite{mouzhang}; M\'esz\'aros and Mou \cite{meszaros2022mean}; Gangbo and M\'esz\'aros \cite{Gangbo2022}) and suggests that the further study of acceleration-free convexity may be applicable in this area.

Since the set of acceleration-free curves includes every generalised geodesic, the notion of convexity along acceleration-free curves is stronger still than the notion of convexity along generalised geodesics. Consequently, in proving Theorem \ref{geodesicequivalencetheorem}, our main focus is to show that every geodesically convex, differentiable function is also convex along acceleration-free curves. As we demonstrate in Example \ref{counterexample}, the notions of geodesic convexity and acceleration-free convexity do not generally coincide, even if we assume our functions to be continuous. Consequently, the assumption of differentiability forms a part of our analysis.

\subsection{Strategy of Proof - Theorem \ref{geodesicequivalencetheorem}.}\label{subsec:one}
In order to clearly understand acceleration-free curves and their associated notion of convexity, our analysis focuses on acceleration-free curves between discrete measures. Firstly, this is because such curves may be viewed, at the finite-dimensional level, as a collection of straight-line paths between a finite collection of points in $\domain$. Secondly, we choose discrete measures because the set of discrete measures on $\domain$ is dense in $(\sP_2(\domain), W_2)$. 

Consider $[0,1]\ni t\mapsto \m_t$, an acceleration-free curve between discrete measures. Since the measure $\m_s$ is also discrete for all $s\in(0,1)$, we may envision the curve $[0,1]\ni t\mapsto \m_t$ as a collection of straight-line paths between a finite collection of points in $\domain$. Moreover, given $s\in (0,1)$, there is no `crossing of mass' between the underlying particles on a small interval on either side of $s$ (even though there may be `crossing of mass' when $t=s$ itself). As we show in Lemma \ref{onesidedgeodesic}, this behaviour means that, for every $s\in (0,1)$, there exist $\varepsilon, \delta > 0 $ such that the curves $[s,s+\varepsilon]\ni t\mapsto \m_t$ and $[s-\delta,s]\ni t\mapsto \m_t$ define geodesics. Consequently, if a function $F\co\sP_2(\domain)\to\R$ is assumed to be geodesically convex, then the map $ t\mapsto F(\m_t)$ must be convex on the intervals $[s,s+\varepsilon]$ and $[s-\delta,s]$. 

In general, convexity on the intervals $[s,s+\varepsilon]$ and $[s-\delta,s]$ does not mean that the map $t\mapsto F(\m_t)$ is convex on $[0,1]$. This is because joining the two intervals $[s-\varepsilon,s]$ and $[s,s+\varepsilon]$ may create a non-convex `cusp' at $t=s$. In order to circumvent the existence of these cusps, we show in Lemma \ref{convexjoininglemma} that no cusps can exist when $[0,1]\ni t\mapsto F(\m_t)$ is differentiable. We also show, in Lemma \ref{accelerationfreederivative}, that $[0,1]\ni t\mapsto F(\m_t)$ is differentiable whenever $F$ is differentiable on $\sP_2(\domain)$. 

Subsequently, in the Proof of Theorem \ref{geodesicequivalencetheorem}, we assume that $F$ is differentiable and $\l$-geodesically convex. As a consequence of the Lemmas \ref{onesidedgeodesic}, \ref{accelerationfreederivative} and \ref{convexjoininglemma}, it follows that $F$ is $\l$-convex along any acceleration-free curve between discrete measures. Moreover, since $F$ is continuous, the convexity criteria, Lemma \ref{empiricalapproximation}, states that $F$ must also be $\l$-convex along any acceleration-free curve.

\subsection{Higher Order Convexity Criteria.} Supplementary to our first result, we present Theorems \ref{convexityone} and \ref{convexitytwo}. These theorems characterise geodesic semi-convexity for differentiable and twice differentiable functions on $\sP_2(\domain)$ respectively. Moreover, as a consequence of Theorem \ref{geodesicequivalencetheorem}, these results also characterise semi-convexity along generalised geodesics and acceleration-free curves. Furthermore, whilst we only address geodesic semi-convexity in this manuscript, we expect that Theorems \ref{convexityone} and \ref{convexitytwo} may be suitably extended to characterise the notion of geodesic $\omega$-convexity introduced by Craig in \cite{Craig2017}.

\begin{theorem}\label{convexityone}
    Let $F\co\sP_2(\domain)\to\R$ and let $\l \in \R$. If $F$ is differentiable on $\sP_2(\domain)$ then $F$ is $\l$-geodesically convex if and only if, for all $\m_1,\m_2$ and all $\g \in \G_o(\m_1,\m_2)$, the following inequality holds. 
\begin{equation}\label{displacementmonotone}
    \int_{{(\domain)}^2}(\nabla_wF[\m_2](x_2)-\nabla_wF[\m_1](x_1))\cdot(x_2-x_1) \ d\g(x_1,x_2)\geq \l W_{2}^2(\m_1,\m_2).
\end{equation}
\end{theorem}
\begin{theorem}\label{convexitytwo}
     Let $F\co\sP_2(\domain)\to\R$ and let $\l \in \R$. If $F$ is twice differentiable on $\sP_2(\domain)$ then $F$ is $\l$-geodesically convex if and only if, for all $\m\in\sP_2(\domain)$ and all $\z \in T_\m\sP_2(\domain)$, the following inequality holds.
\begin{equation}\label{displacementmonotonetwo}
HessF[\m](\z,\z) \geq \l \| \z \| _{L^2(\m)}^2 
\end{equation}
\end{theorem}

 Theorems \ref{convexityone} and \ref{convexitytwo} are proven in \hyperref[sec:two]{Section Four} and may be seen as analogues of the respective inequalities that characterise semi-convexity for differentiable and twice differentiable functions on finite-dimensional space. Although the $2$-Wasserstein space is not a smooth manifold, the formal Riemannian calculus, proposed by Otto in \cite{Otto2001} and further developed by Otto and Villani in \cite{Otto2000}, provides a well-established notion of differentiability for functions defined on $\sP_2(\domain)$. For the notions of gradient and Hessian utilised in this manuscript, we refer to the theory of differential $1$-forms and $2$-forms subsequently developed by Gangbo et al. in \cite{Gangbo2008}. 
In particular, the work of Gangbo and Chow (see \cite{Chow2019}) elucidates that the differential 2-form proposed in \cite{Gangbo2008} defines a Hessian on the $2$-Wasserstein space and that this Hessian is consistent with the Levi-Civita connection proposed by Gigli in \cite{Gigli2009} and Lott in \cite{Lott2008}. 

Whilst the current theory is lacunary in a complete presentation of the first and second-order convexity criteria presented above, there are well-established first-order-convexity inequalities such as the aforementioned notions of \emph{L-convexity} and \emph{displacement monotonicity}. 
Furthermore, Lanzetti et al. show that, when a function $F\co\sP_2(\domain)\to\R$ is differentiable and geodesically semi-convex, Equation \eqref{displacementmonotone} holds for all $\m_1,\m_2 \in \sP_2(\domain)$ and all $\g \in \G_o(\m_1,\m_2)$ (cf. \cite[Proposition 2.8]{lanzetti2022}). Comparatively, in Theorem \ref{convexityone}, we posit that the converse also holds.
 
In contrast to the first-order convexity criteria, second-order convexity inequalities, such as \eqref{displacementmonotonetwo}, are far less well-established. In particular, it seems there is currently no description of geodesic semi-convexity purely in terms of the Hessian presented in \cite[Definition 3.1]{Chow2019} (see also Definition \ref{Hessiandefn}). In addition, as we explain in the following subsection, the nature of parallel transport on $2$-Wasserstein space means that establishing a second-order characterisation of geodesic convexity requires particular care.

\subsection{Strategy of Proof - Theorem \ref{convexitytwo}.}\label{subsec:two}
As seen in Definition \ref{Hessiandefn}, the Wasserstein Hessian is defined via its extension from $\nabla C_c^\infty(\domain)\times \nabla C_c^\infty(\domain)$, however, it is also possible to calculate the Hessian directly via the covariant derivative proposed in \cite{Gigli2009} and \cite{Lott2008}. 
Whilst it is expected that one could employ this latter method to derive a second-order geodesic convexity criterion, there are some difficulties in this approach. In particular, as established in \cite[Example 5.20]{Gigli2009}, parallel transport does not exist everywhere along some Wasserstein geodesics. In recognition of this difficulty, we instead choose to establish a second-order geodesic convexity criterion via an alternative argument which we describe as follows. 

Firstly, let $[0,1]\ni t \mapsto \m_t \in \sP_2(\domain)$ be a geodesic between measures $\m,\n \in \sP_2(\domain)$ and assume that there exists $\varphi \in C_c^\infty(\domain)$ such that $\m_t = (id + t\nabla\varphi)_{\#}\m$ for all $t \in [0,1]$. 
In Lemma \ref{secondderivative}, we show (under some additional convexity assumption on $\varphi$) that the second derivative of the map $t\mapsto F(\m_t)$ may be written explicitly in terms of the Wasserstein Hessian. Consequently, convexity along any such curve $t\mapsto\m_t$ may be characterised by Inequality \eqref{displacementmonotonetwo}.  

In order to extend this characterisation to a larger class of geodesics we subsequently introduce the set $\sP^{rc}_2(\domain)$. This is the subset of $\sP_2(\domain)$ containing measures which are absolutely continuous with respect to the $d$-dimensional Lebesgue measure and have compact support. In particular, in Lemma \ref{approximationlemma}, we show that, if $\m,\n \in \sP_2^{rc}(\domain)$, then the optimal map between any $\m$ and $\n$ can be approximated by a sequence of smooth optimal maps $(T_n)_{n\in\N}$ for which $(T_n -id) \in \nabla C_c^\infty(\domain)$. The proof of this result is inspired by the mollifying argument used to prove \cite[Proposition 8.5.2]{Ambrosio2008}.

As a consequence of Lemmas \ref{secondderivative} and \ref{approximationlemma}, it follows that Inequality \eqref{displacementmonotonetwo} characterises geodesic convexity on $\sP_2^{rc}(\domain)$. Moreover, as a consequence of Lemma \ref{convexitycriterion}, Inequality \eqref{displacementmonotonetwo} characterises geodesic convexity on the whole space $\sP_2(\domain)$. Lemma \ref{convexitycriterion} states that a continuous function is $\l$-geodesically convex if and only if its restriction to $\sP_2^{rc}(\domain)$ is $\l$-geodesically convex and is inspired by the convexity criterion derived in \cite[Proposition 9.1.3]{Ambrosio2008}.

\subsection{A Note.}
After the initial submission of this manuscript to the arXiv, the author was made aware of the work of Cavagnari, Savar\'e and Sodini \cite{cavagnari} which was developed in parallel with this manuscript. The results of \cite{cavagnari} present several connections with \hyperref[sec:one]{Section Three} of this manuscript. In particular, the authors introduce the notion of \emph{total convexity} and this corresponds to our notion of \emph{convexity along acceleration-free curves}. 

In addition to this, it is shown in \cite[Theorem 9.1; Remark 9.2]{cavagnari} that, if $d\geq 2$, then the assumption of differentiability in Theorem \ref{geodesicequivalencetheorem} may be relaxed to the assumption of continuity. To prove this result, the authors of \cite{cavagnari} also find the space of discrete measures crucial to their argument and, moreover, the findings of Lemma \ref{onesidedgeodesic} may similarly be found in \cite[Theorem 6.2]{cavagnari}.

%% file: chapters/notation.tex
\section{Notation and Preliminaries}
\subsection{Notation.}
The notation introduced here is largely consistent with the notation used in \cite{Ambrosio2008}. We refer the reader to \cite{Ambrosio2008} for a more detailed description of their properties. 

The set $\sP(\domain)$ denotes the space of Borel probability measures over $\domain$ and the set $\sP_2(\domain)$ denotes the set of $\m \in \sP(\domain)$ with bounded second moment. 
For $\m,\n \in \sP_2(\domain)$, the set of transport plans $\G(\m,\n)$ denotes the set of $\g \in \sP_2(\domain\times\domain)$ with first marginal $\m$ and second marginal $\n$. 
Subsequently, we define the $2$-Wasserstein distance $W_2 \colon \sP_2(\domain)\times\sP_2(\domain) \to \R$.
\[
W_2(\m,\n) :=\inf\bigg\{\bigg(\int_{(\domain)^2}|x_1-x_2|^2 \ d\g (x_1,x_2)\bigg)^\frac{1}{2}\bigg| \ \g \in \G(\m,\n) \bigg\}.
\]
The pair $(\sP_2(\domain), W_2)$ defines a metric space which we refer to as $2$-Wasserstein space and, throughout this manuscript, we will assume that $\sP_2(\domain)$ is endowed with the $W_2$ distance. 

The set of optimal transport plans $\G_o(\m,\n)$ is defined as follows.
\[
\G_o(\m,\n) := \bigg\{\g\in \G(\m,\n) \bigg| \int_{(\domain)^2}|x_1-x_2|^2 \ d\g (x_1,x_2) = W_2^2(\m,\n) \bigg\} .
\]
Given $\m \in \sP_2(\domain)$, we say that $\m$ is a discrete measure if there exists $n \in \N$ and $a_1,\dots,  a_n\in [0,1]$ such that $\sum_{i=1}^n a_i = 1$ and $x_1,\dots, x_n \in \domain$ such that $\m = \sum_{i=1}^n a_i\delta_{x_i}$. We let $\sP_2^{d}(\domain)$ denote the set of discrete measures on $\domain$. The set $\sP_2^r(\domain)$ denotes the set of $\m \in \sP_2(\domain)$ that are absolutely continuous with respect to the $d$-dimensional Lebesgue measure $\mathscr{L}^d$ and $\sP_2^{rc}(\domain)$ denotes the set of $\m \in \sP_2^r(\domain)$ with compact support. 
For $\m\in \sP(\domain)$, the set $L^2(\m;\domain)$ denotes the Hilbert space of Borel vector fields $\z\co \domain \to\R^d$ satisfying $\int_{\domain} |\z|^2 \ d\m < \infty$.  
We define $\nabla C_c^\infty(\domain) := \{\nabla\varphi \ | \ \varphi \in  C_c^\infty(\domain)\}$ and, for $\m \in \sP_2(\domain)$, the set $T_\m\sP_2(\domain)$ denotes the closure of $\nabla C_c^\infty(\domain)$ in $L^2(\m;\domain)$. 
For ${i,j}\in\{1,2,3\}$ and ${k}\in\{1,2\}$, we define the projection operators $\pi^{i,j}(x_1,x_2,x_3):=(x_i,x_j)$ and $\pi^{k}(x_1,x_2):=x_k$. 
Given a set $A\subset \domain$, its convex hull is denoted $\Hull(A)$. Finally, given a measure $\m\in \sP(\domain)$, its support is denoted $\supp(\m)$ and, given a map $\m$-measurable map $T\co\domain \to \R^n$, the pushforward of $\m$ through $T$ is the measure $T_{\#}\m\in \sP(\R^n)$ defined by $T_{\#}\m(A) := \m(T^{-1}(A))$ for all Borel sets $A \subset \R^n$. 

%% file: chapters/preliminaries.tex
\subsection{Preliminaries.}\label{preliminaries}
In this subsection, we recall a number of established definitions and results for use throughout the remainder of the manuscript. In particular, since Theorem \ref{geodesicequivalencetheorem} concerns both geodesic convexity and convexity along generalised geodesics, we first recall the definitions of such concepts from \cite[Definition 1.1]{McCann1997} and \cite[Definition 9.2.4]{Ambrosio2008} respectively. Furthermore, we recall from \cite[Lemma 7.2.1]{Ambrosio2008}, a useful lemma concerning the properties of Wasserstein geodesics.
\begin{defn}\label{geodesicconvexity}
We say that a curve $[0,1]\ni t\mapsto\m_t\in \sP_2(\domain)$ is a \emph{geodesic} if there exist $\m_1,\m_2\in\sP_2(\domain)$ and $\g \in \G_o(\m_1,\m_2)$ such that $\m_t = ((1-t)\p^1 + t\p^2)_{\#}\g$ for all $t\in [0,1]$. 
Given $F\co\sP_2(\domain)\to\R$ and $\l \in \R$, we say that $F$ is \emph{$\l$-geodesically convex} if, between every pair of measures $\m_1,\m_2 \in \sP_2(\domain)$, there exists a geodesic $[0,1]\ni t\mapsto\m_t\in \sP_2(\domain)$ such that the following inequality is satisfied for all $t\in[0,1]$. \begin{equation}\label{geodesicconvexityinequality}
    F(\m_t) \leq (1-t)F(\m_1) + tF(\m_2) - \frac{\l}{2}t(1-t) W_2^2(\m_1,\m_2).  
    \end{equation}
\end{defn}

\begin{defn}\label{generalisedgeodesicconvexity}
  Let $F\co\sP_2(\domain)\to\R$ and let $\l \in \R$. We say that $F$ is \emph{$\l$-convex along generalised geodesics} if, for every triplet of measures $\m_1,\m_2,\m_3 \in \sP_2(\domain)$, there exists $\boldsymbol{\m}\in \sP_2((\domain)^3)$ with $\pi^{1,2}_{\#}\boldsymbol{\m} \in \G_o(\m_1,\m_2)$ and $\pi^{1,3}_{\#}\boldsymbol{\m} \in \G_o(\m_1,\m_3)$ such that the curve $[0,1]\ni t\mapsto \m_t:= ((1-t)\p^2 + t\p^3)_{\#}\boldsymbol{\m}$ satisfies the following inequality for all $t\in[0,1]$.
  \[
  F(\m_t) \leq (1-t)F(\m_2) + tF(\m_3) -\frac{\l}{2}t(1-t)\int_{(\domain)^3}|x_2-x_3|^2 \ d\boldsymbol{\m}(x_1,x_2,x_3).
  \]
\end{defn}

\begin{lemma}\label{geodesicremark}
Let $\m_1,\m_2\in \sP_2(\domain)$, let $\g\in \G_o(\m_1,\m_2)$ and let $[0,1]\ni t\mapsto \m_t := ((1-t)\p^1 + t\p^2)_{\#}\g$. The set $\G_o(\m_s,\m_r)$ has a unique element for $\{s,r\}\neq \{0,1\}$. 
In particular, this element is given by $[((1-s)\p^1 + s\p^2),((1-r)\p^1 + r\p^2)]_{\#}\g$. 
\end{lemma}

\begin{remark}\label{everygeodesicremark}
If $F\co\sP_2(\domain)\to\R$ is $\l$-geodesically convex then $F$ necessarily satisfies Equation \eqref{geodesicconvexityinequality} for at least one geodesic between every pair of measures in $\sP_2(\domain)$.
However, if $F$ is $\l$-geodesically convex and continuous, then $F$ necessarily satisfies Equation \eqref{geodesicconvexityinequality} for every geodesic between every pair of measures in $\sP_2(\domain)$. 
This is a consequence of Lemma \ref{geodesicremark}.
\end{remark}

In order to establish Theorems \ref{convexityone} and \ref{convexitytwo}, we first require notions of differentiability and twice differentiability on the $2$-Wasserstein space. Subsequently, we recall these concepts from \cite[Definition 4.9]{Gangbo2008} and \cite[Definition 3.1]{Chow2019} respectively.
\begin{defn}\label{differentiability}
Let $F\co\sP_2(\domain)\to\R$. 
Given $\m \in \sP_2(\domain)$ and $\xi \in L^2(\m;\domain)$, we say $\xi \in \partial F[\m]$ if 
\[
F(\n) - F(\m) = \int_{(\domain)^2} \xi(x_1) \cdot (x_2-x_1) \ d\g(x_1,x_2) + o(W_2(\m,\n))
\]
for all $\n \in \sP_2(\domain)$ and all $\g \in \G_o(\m,\n)$. 
If $\partial F[\m]$ is non-empty, then we say that $F$ is differentiable at $\m$ and define the Wasserstein gradient $\nabla_wF[\m]$ to be its element of minimal norm. 
We also define the differential $dF[\m]:L^2(\m;\domain)\to\R$, $dF[\m](\z) := \zeta \cdot F[\m] := \langle \zeta, \nabla_wF[\m]\rangle_{L^2(\m)}.$
\end{defn}

\begin{defn}\label{Hessiandefn}
    Let $F$ be differentiable in a neighbourhood of $\m$ and let $\n \mapsto \z \cdot F(\n)$ be differentiable at $\m$ for all $\zeta \in \nabla C_c^\infty(\domain)$.
    We define $\bar{H}ess F[\m] \co \nabla C_c^\infty(\domain)\times \nabla C_c^\infty(\domain) \to \R$ as follows. 
    \[
    \bar{H}ess F[\m](\z_1,\z_2) := \z_1 \cdot (\z_2 \cdot F[\m]) - (\nabla\z_2 \ \z_1) \cdot F[\m]. 
    \]
    If $\bar{H}ess F[\m]$ exists and there exists $C\in \R$ such that $\bar{H}ess F[\m](\z_1,\z_2) \leq C \| \z_1 \| _{L^2(\m)}\| \z_2 \| _{L^2(\m)}$ for all $\z_1,\z_2 \in \nabla C_c^\infty(\domain)$ then $\bar{H}ess F[\m]$ has a unique extension onto $T_\m\sP_2(\domain)\times T_\m\sP_2(\domain)$. 
    In addition, we denote this extension by $Hess F[\m]$ and say that $F$ is twice differentiable at $\m$.
\end{defn}

Throughout this manuscript, we differentiate functions of the form $f\co[a,b]\to\R$. 
When we say that $f$ is differentiable on $[a,b]$, we mean that $f$ is differentiable on $(a,b)$ and its respective right and left-sided derivatives exist at the points $a$ and $b$. 
When evaluated at a generic point on the interval $[a,b]$, we denote derivatives of $f\co[a,b]\to\R$ by $f'$ or $\frac{df}{dx}$, however, when evaluated specifically at the endpoints, we utilise the right and left-sided derivatives which we denote by $\frac{df}{dx}_+$ and $\frac{df}{dx}_-$ respectively. Using this definition of differentiability, we derive the following convexity criteria which we will use to prove Theorem \ref{convexityone}.

\begin{lemma}\label{monotonederivative}
Let $f\co[a,b]\to\R$ and let $\l\in \R$. If $f$ is differentiable on $[a,b]$ then $f$ is $\l$-convex if and only if it satisfies $(f'(x)-f'(y)) \geq \l(x-y) $ for all $a\leq x,y \leq b$.
\end{lemma}
\begin{proof}
    Since $x\mapsto f(x)$ is $\l$-convex if and only if $x\mapsto f(x)-\frac{\l}{2}x^2$ is convex, it suffices to show that $f$ is convex if and only if $(f'(x)-f'(y))(x-y) \geq 0 $ for all $a\leq x,y \leq b$. Furthermore, since $f$ is assumed to be differentiable, it is also continuous. Consequently, it suffices to show that $f$ is convex on $(a,b)$.

    It is a well-known result in convex analysis that $f\co[a,b]\to\R$ is convex on $(a,b)$ if and only if $f$ satisfies $(f'(y)-f'(x))(y-x) \geq 0$  for all $a < x,y < b$ (cf. \cite[Theorem 12.18]{Binmore}). Furthermore, it is shown in \cite[Appendix C, Theorem 1]{Pollard}, that the left and right-sided derivatives of a convex function satisfy the following inequality
    \[
     (b-a) \frac{df}{dx}_{-}\bigg|_{x=b}\geq f(b)-f(a) \geq (b-a)\frac{df}{dx}_{+}\bigg|_{x=a}.
    \]
    Identifying the one-sided derivatives at $b$ and $a$ with $f'(b)$ and $f'(a)$ respectively, the above inequality implies that $(f'(b) - f'(a))(b-a) \geq (f(b)-f(a)) - (f(b)-f(a)) = 0$. Consequently, it follows that $f$ is convex on $[a,b]$ if and only $f$ satisfies $(f'(x)-f'(y))(x-y) \geq 0 $  for all $a \leq x,y \leq b$.
\end{proof}

%% file: chapters/chapterone.tex
\section{Convexity Along Acceleration-Free Curves}
\label{sec:one}
In this section, our goal is to prove Theorem \ref{geodesicequivalencetheorem}. To achieve this, we first introduce acceleration-free curves and their associated notion of convexity. Subsequently, we examine some properties of these curves and develop a number of preparatory lemmas concerning the differentiation of functions along such curves. Finally, we prove Theorem \ref{geodesicequivalencetheorem} and construct an example to show that, in general, the notions of geodesic convexity and acceleration-free convexity do not coincide, even if we assume that our functions are continuous.
\subsection{Acceleration-Free Curves} 
\begin{defn}\label{accelerationfree}
We say that a curve $[0,1]\ni t\mapsto \m_t\in\sP_2(\domain)$ is \emph{acceleration-free} if there exist $\m_1,\m_2\in\sP_2(\domain)$ and $\g \in \G(\m_1,\m_2)$ such that $\m_t = ((1-t)\p^1 + t\p^2)_{\#}\g$ for all $t\in [0,1]$.  
Given $F\co\sP_2(\domain)\to\R$ and $\l \in \R$, we say that $F$ is \emph{$\l$-convex along acceleration-free curves} if, between every pair $\m_1,\m_2 \in \sP_2(\domain)$, Equation \eqref{convexityinequality} is satisfied for all $\g \in \G(\m_1,\m_2)$ and all $t\in[0,1]$. 
\begin{equation}\label{convexityinequality}
F(((1-t)\p^1 + t\p^2)_{\#}\g) \leq (1-t)F(\m_1) + tF(\m_2) -\frac{\l}{2}t(1-t)\int_{(\domain)^2}|x_1-x_2|^2 \ d\g(x_1,x_2). 
\end{equation}
\end{defn}

\begin{remark}
An acceleration-free curve may be induced by any transport plan, however, a Wasserstein geodesic is an acceleration-free curve induced by a transport plan which is optimal for the transport between the initial and final measures.
\end{remark}

\begin{remark}\label{lambdaconvexitycriterion}
 A function $F\co\sP_2(\domain)\to\R$ is $\l$-convex along acceleration-free curves (resp. $\l$-geodesically convex) if and only if the map $\sP_2(\domain)\ni \m \mapsto F(\m) - \frac{\l}{2} \int_{\domain} |x|^2 d\m$ is convex along acceleration-free curves (resp. geodesically convex).
\end{remark}

In the following Lemma, we show that an acceleration-free curve between discrete measures also defines a Wasserstein geodesic when we restrict ourselves to small enough sub-intervals of $[0,1]$. This property is key to our further analysis of acceleration-free convexity as it allows us to study acceleration-free curves using the properties of Wasserstein geodesics.
\begin{lemma}\label{onesidedgeodesic}
Let $\m_1,\m_2 \in \sP_2^d(\domain)$, let $\g \in \G(\m_1,\m_2)$ and let $[0,1] \ni t\mapsto \m_t:= ((1-t)\pi^1 + t\pi^2)_{\#}\g$.
For every $s \in [0,1)$, there exists $\varepsilon > 0$ such that, up to a time re-scaling, $[s,s+\varepsilon]\ni t \mapsto \m_t$ defines a unit time geodesic between $\m_s$ and $\m_{s+\varepsilon}$.  
Likewise, for every $s \in (0,1]$, there exists $\delta > 0$ such that, up to a time re-scaling, $[s-\delta,s]\ni t \mapsto \m_t$ defines a unit time geodesic between $\m_{s-\delta}$ and $\m_s$. 
\end{lemma}

\begin{proof}
Since $\m_1,\m_2$ are discrete measures, we may denote $\m_1 = \sum_{i=1}^n a_i\delta_{x_i}$ and $\m_2 =  \sum_{j=1}^m b_j\delta_{y_j}$ in such a way that $x_i \neq x_k$ for $i \neq k$ and $y_j\neq y_l$ for $j \neq l$. 
Additionally, since $\g\in \G(\m_1,\m_2)$, it follows from the marginal properties of $\g$ that there exist non-negative coefficients $c_{ij}$ satisfying $\sum_{i,j=1}^{n,m} c_{ij}=1$ and such that $\m_t = \sum_{i,j=1}^{n,m} c_{ij}\delta_{(1-t)x_i + ty_j}$ for all $ t\in [0,1]$. Define $l := nm$ and, for $i,j \in \{1,\dots,n\}\times \{1,\dots,m\}$, define 
\[
\theta_{(i-1)m+j} := c_{ij}, \ z_{(i-1)m+j} := y_j -x_i,\  w_{(i-1)m+j} := x_i.
\] 
It follows that $\m_t = \sum_{i,j=1}^{n,m} c_{ij}\delta_{(1-t)x_i + ty_j} = \sum_{k=1}^{l} \theta_k\delta_{z_kt + w_k}$. 

Fix $p,q\in  \{1,\dots,l\}$. Since the elements of $x_1,\dots,x_n$ and $y_1,\dots,y_m$ were chosen to be distinct, there exists at most one $t\in[0,1]$ such that $z_pt + w_p = z_qt + w_q$.  Consequently, there exists a  finite set $Q \subset [0,1]$ such that $z_pt + w_p \neq z_qt + w_q$ for all $p,q\in  \{1,\dots,l\}$ and $t\in [0,1]\setminus Q$.

Fix $s\in[0,1)$, fix $p \in \{1,\dots,l\}$ and let $P_{p}^s\subset  \{1,\dots l\}$ denote the set of integers $q$ such that $z_ps + w_p = z_qs + w_q$. In particular, we remark that, $p\in P_p^s$ for all $p \in \{1,\dots l\}$ and, consequently, the set $P_p^s$ is always non-empty. Since the set $Q$ is finite, we may choose $\varepsilon_{p}>0 $ such that $(s,s+\varepsilon_{p}]\cap Q =\varnothing$. Consequently, there exists $\d_{p}>0$ such that $
|z_ps + w_p - (z_qt + w_q)|>\delta_{p}$ for all $q\notin P_{p}^s$ and all $t\in [s,s+\varepsilon_{p}]$. Additionally, for any $q \in P^s_{p}$, it follows that 
\[
\lim_{t\to s} |z_ps + w_p - (z_qt + w_q)| = 0.
\] 
Consequently, we may re-choose an even smaller $\varepsilon_{p}> 0$ such that
$\delta_{p}
> |z_ps + w_p - (z_qt + w_q)|$
for all $ t\in [s,s+\varepsilon_{p}]$ and all $q \in P_{p}^s$. In particular, by re-choosing a smaller $\varepsilon_{p}$ it still holds that
$|z_ps + w_p - (z_qt + w_q)|>\delta_{p}$ for all $q\notin P_{p}^s$ and all $t\in [s,s+\varepsilon_{p}]$. Moreover, Equation \eqref{infimum} holds for all $t\in [s,s+\varepsilon_{p}]$ and all $q \in P_{p}^s$.  
\begin{equation}
\label{infimum}
\begin{split}
\inf_{q'\notin P_{p}^s} & |z_ps + w_p - (z_{q'}t + w_{q'})| > \delta_p
>  |z_ps + w_p - (z_{q}t + w_{q})| = |z_ps + w_p - (z_{q}t + w_{p})|.
\end{split}
\end{equation}
We subsequently define $\varepsilon := \inf_{1\leq p\leq l} \varepsilon_p$.  

\textbf{Claim:} Fix $t \in [s, s+\varepsilon]$. There exists an optimal plan $\s\in \G_o(\m_s,\m_t)$ of the form $\s:= \sum_{k=1}^{l}\theta_k (\delta_{z_ks + w_k}\otimes \delta_{z_kt + w_k})$. 
\begin{proof}[Proof of Claim]
To show that $\s$ is an optimal plan it is sufficient to show that the support of $\s$ is cyclically monotone. This result is a consequence of \cite[Theorem 6.1.4]{Ambrosio2008}. Moreover, the support of $\s$ is cyclically monotone if the following inequality is satisfied any permutation $\rho$ of the integers $\{1,\dots,l\}$. 
\[
\sum_{k=1}^l |z_ks + w_k-(z_{\rho(k)}t + w_{\rho(k)})|^2 \geq \sum_{k=1}^l |z_ks + w_k-(z_{k}t + w_{k})|^2
\]
We first assume that $\rho$ is a permutation map such that $\rho(k) \in P_k^s$ for all $k \in \{1,\dots,l\}$. By the definition of $P_k^s$ it follows that $z_ks + w_k =z_{\rho(k)}s + w_{\rho(k)}$ for all $k \in \{1,\dots,l\}$. consequently, the following equality holds.
\[
\sum_{k=1}^l |z_ks + w_k-(z_{\rho(k)}t + w_{\rho(k)})|^2 = 
\sum_{k=1}^l |z_{\rho(k)}s + w_{\rho(k)}-(z_{\rho(k)}t + w_{\rho(k)})|^2 = \sum_{k=1}^l |z_ks + w_k-(z_{k}t + w_{k})|^2.
\]
On the other hand, assume that $\rho$ is a permutation map and there exists a set $R \subset \{1,\dots,l\}$ such that $\rho(k) \notin P_k^s$ for all $k\in R$ and such that $\rho(k) \in P_k^s$ for all $k\in \{1,\dots,l\}\setminus R$. As a consequence of Equation \eqref{infimum}, the following system of inequalities holds.
\begin{flalign*}
& \sum_{k=1}^l |z_ks + w_k-(z_{\rho(k)}t + w_{\rho(k)})|^2\\ = 
& \sum_{k\in R} |z_ks + w_k-(z_{\rho(k)}t + w_{\rho(k)})|^2 + \sum_{k\notin R} |z_ks + w_k-(z_{\rho(k)}t + w_{\rho(k)})|^2 \\ > & \sum_{k\in R} \delta_k + \sum_{k\notin R} |z_ks + w_k-(z_{\rho(k)}t + w_{\rho(k)})|^2 = \sum_{k\in R} \delta_k + \sum_{k\notin R} |z_{\rho(k)}s + w_{\rho(k)}-(z_{\rho(k)}t + w_{\rho(k)})|^2 
\\
> & \sum_{k=1}^l |z_{\rho(k)}s + w_{\rho(k)}-(z_{\rho(k)}t + w_{\rho(k)})|^2 = \sum_{k=1}^l |z_{k}s + w_{k}-(z_{k}t + w_{k})|^2.
\end{flalign*}
Since the two cases we have considered include all possible permutations of the integers $\{1,\dots,l\}$, we conclude that the support of $\s$ must be cyclically monotone and, hence, the plan $\s$ must be optimal between $\m_s$ and $\m_t$.
\end{proof}
Since $\s$ defines an optimal plan,  Equation \eqref{discretewassersteindistance} holds.
\begin{equation}\label{discretewassersteindistance}
\begin{split}
W_2(\m_s, \m_t) & = \bigg(\int_{(\domain)^2} |x-y|^2 \ d\s(x,y)\bigg)^\frac{1}{2} 
\\
& =  \bigg(\sum_{k=1}^{l} \theta_k |z_ks + w_k-(z_kt + w_k)|^2\bigg)^\frac{1}{2} = |t-s| \bigg(\sum_{k=1}^{l} \theta_k |z_k|^2\bigg)^\frac{1}{2} .
\end{split}
\end{equation} 
As a consequence of Equation \eqref{discretewassersteindistance}, the following equality also holds for all $t \in [s,s+\varepsilon]$. 
\[
\frac{W_2(\m_s,\m_t)}{|t-s|}= \bigg(\sum_{k=1}^{l} \theta_k |z_k|^2\bigg)^\frac{1}{2} = \frac{W_2(\m_s,\m_{s+\varepsilon})}{\varepsilon}.
\] 
Moreover, it follows that $\varepsilon W_2(\m_s,\m_t) = |t-s|W_2(\m_s,\m_{s+\varepsilon})$ for all $t \in [s,s+\varepsilon]$, and so, up to a time re-scaling of a factor $\varepsilon$, the curve $[s,s+\varepsilon]\ni t \mapsto \m_t$ defines a unit time geodesic between $\m_s$ and $\m_{s+\varepsilon}$. By a similar argument, for every $s \in (0,1]$, there exists $\delta > 0$ such that, up to a time re-scaling, $[s-\delta,s]\ni t \mapsto \m_t$ defines a unit time geodesic between $\m_{s-\delta}$ and $\m_s$.    
\end{proof}
To further motivate the study of acceleration-free curves between discrete measures, the following lemma shows that it is enough, in the context of proving Theorem \ref{geodesicequivalencetheorem}, to show that a function is convex along acceleration-free curves when restricted to the set $\sP_2^d(\domain)$.

\begin{lemma}\label{empiricalapproximation}
Let $F\co\sP_2(\domain)\to\R$. If $F$ is continuous then $F$ is $\l$-convex along acceleration-free curves if and only if its restriction to $\sP^d_2(\domain)$ is $\l$-convex along acceleration-free curves.
\end{lemma}
\begin{proof}
    If $[0,1]\ni t\mapsto\m_t$ defines an acceleration-free curve between two discrete measures then $\m_t$ is a discrete measure for all $t\in[0,1]$. Consequently, if $F$ is $\l$-convex along acceleration-free curves then its restriction to $\sP^d_2(\domain)$ is $\l$-convex along acceleration-free curves and it is left to prove the converse implication. Let $\m_1, \m_2 \in \sP_2(\domain)$ and let $\g \in \G(\m_1,\m_2)$.
    Since $\sP_2^d(\domain\times\domain)$ is dense in $\sP_2(\domain\times\domain)$, there exists a sequence of discrete measures $(\g_n)_{n\in\N}$ such that $\lim_{n\to\infty}W_2(\g_n,\g) = 0$. Defining the curves $[0,1]\ni t\mapsto \m_t^n := ((1-t)\pi^1 +t \pi^2)_{\#}\g_n$ and $[0,1]\ni t\mapsto \m_t := ((1-t)\pi^1 +t \pi^2)_{\#}\g$, it follows from the convergence of  $(\g_n)_{n\in\N}$ that $\lim_{n\to\infty}W_2(\m_t^n,\m_t) = 0$ for all $t\in[0,1]$. 
    Moreover, since $\g_n \in \sP_2^d(\domain\times\domain)$, for every $n\in \N$, the curve $[0,1]\ni t\mapsto \m_t^n$ defines an acceleration-free curve between two discrete measures which we denote $\m_{1,n}$ and $\m_{2,n}$. Consequently, if we assume that $F\co\sP_2(\domain)\to\R$ is continuous and its restriction to $\sP_2^d(\domain)$ is $\l$-convex along acceleration-free curves, then the following system of inequalities holds for all $t\in[0,1]$.
    \begin{flalign*}
       F(\m_t)&  = \lim_{n\to\infty} F(\m_t^n)  \\
       & \leq (1-t)\lim_{n\to\infty} F(\m_{1,n})  + t\lim_{n\to\infty} F(\m_{2,n}) - t(1-t)\frac{\l}{2}\lim_{n\to\infty}  \int_{(\domain)^2} |x_1-x_2|^2 \ d\g_n(x_1,x_2) \\
       &= (1-t) F(\m_{1}) + tF(\m_2) - t(1-t)\frac{\l}{2} \int_{(\domain)^2} |x_1-x_2|^2 \ d\g(x_1,x_2).
       \end{flalign*}
       Since $\m_1,\m_2$ and $\g \in \G(\m_1,\m_2)$ were chosen arbitrarily, we conclude that $F$ is $\l$-convex along acceleration-free curves.
\end{proof}

\subsection{Differentiation Along Acceleration-Free Curves and Proof of Theorem \ref{geodesicequivalencetheorem}.}
Before attempting to calculate the derivative of a function along an acceleration-free curve, it is first useful to establish the following lemma which characterises the derivative of a differentiable function along geodesics. Since acceleration-free curves between discrete measures behave somewhat like Wasserstein geodesics, we may also use Lemma \ref{firstderivative} in order to characterise the derivative of a differentiable function along these curves.

\begin{lemma}\label{firstderivative}
    Let $F\co\sP_2(\domain)\to\R$ and let $[0,1]\ni t\mapsto \m_t := ((1-t)\pi^1 + t\pi^2)_{\#}\g$ with $\m_1,\m_2\in\sP_2(\domain)$, and $\g \in \G_o(\m_1,\m_2)$. If $F$ is differentiable then $[0,1]\ni t \mapsto F(\m_t)$ is differentiable on $[0,1]$. In particular, 
    \[
    \frac{d}{dt}F(\m_t) = \int_{(\domain)^2} \nabla_wF[\m_t]((1-t) x_1 +tx_2)\cdot(x_2-x_1) \ d\g (x_1,x_2). 
    \]
 \end{lemma}
\begin{proof} Let $\g_{h,t}$ denote the unique element of $\G_o(\m_t,\m_{t+h})$ characterised by Lemma \ref{geodesicremark}.
\begin{equation*}
\begin{split}
\frac{d}{dt} F(\m_t) & = \lim_{h\to 0} \frac{1}{h}(F(\m_{t+h}) - F(\m_{t})) \\
& = \lim_{h\to 0} \frac{1}{h} \int_{(\domain)^2} \nabla_wF[\m_t](y_1)\cdot(y_2-y_1) \ d\g_{h,t} (y_1,y_2) +  \lim_{h\to 0} \frac{o(W_2(\m_t,\m_{t+h}))}{h} \\
& = \lim_{h\to 0} \frac{h}{h} \int_{(\domain)^2} \nabla_wF[\m_t]((1-t) x_1 +tx_2)\cdot(x_2-x_1) \ d\g (x_1,x_2) + 0 \\
& =  \int_{(\domain)^2} \nabla_wF[\m_t]((1-t) x_1 +tx_2)\cdot(x_2-x_1) \ d\g (x_1,x_2).
\end{split}
\end{equation*}
In particular, since $t\mapsto \m_t$ is a geodesic, the following system of equalities holds.
\[
\lim_{h\to 0} \frac{o(W_2(\m_t,\m_{t+h}))}{h}  = \lim_{h\to 0} \frac{ o(W_2(\m_t,\m_{t+h})) }{ W_2(\m_t,\m_{t+h})}\cdot \lim_{h\to 0}  \frac{W_2(\m_t,\m_{t+h})}{h} = 0 \cdot W_2(\m_1,\m_2) = 0.
\]
\end{proof}

\begin{lemma}\label{accelerationfreederivative}
Let $F\co\sP_2(\domain)\to\R$ and let $[0,1] \ni t\mapsto \m_t := ((1-t)\pi^1 + t\pi^2)_{\#}\g$ with $\m_1,\m_2 \in \sP_2^d(\domain)$ and $\g \in \G(\m_1,\m_2)$. If $F$ is differentiable on $\sP_2(\domain)$ then the map $[0,1]\ni t \mapsto F(\m_t)$ is differentiable on $[0,1]$. In particular,
 \[
    \frac{d}{dt}F(\m_t) = \int_{(\domain)^2} \nabla_wF[\m_t]((1-t) x_1 +tx_2)\cdot(x_2-x_1) \ d\g (x_1,x_2). 
\]
\end{lemma}
\begin{proof} 
Given $s \in [0,1)$, it follows from Lemma \ref{onesidedgeodesic} that there exists $\varepsilon > 0$ such that $[s,s+\varepsilon]\ni t\mapsto \m_t$ defines a geodesic. Moreover, by Definition \ref{geodesicconvexity}, there exists an optimal plan $\g_{s,\varepsilon}\in \G_o(\m_s,\m_{s+\varepsilon})$ such that $\m_{s+t\varepsilon} = ((1-t)\pi^1 + t\pi^2)_{\#}\g_{s,\varepsilon}$ for all $t \in [0,1]$. Consequently, it follows from Lemma \ref{firstderivative} that
\[
 \frac{d}{dt}_{+}F(\m_t)\bigg|_{t=s} = \frac{1}{\varepsilon}\frac{d}{dt}_{+}F(\m_{s+t\varepsilon})\bigg|_{t=0} = \frac{1}{\varepsilon}\int_{(\domain)^2} \nabla_wF[\m_s](y_1)\cdot(y_2-y_1) \ d\g_{s,\varepsilon} (y_1,y_2). 
\]
Since $((1-(s+t\varepsilon))\pi^1 + (s+t\varepsilon)\pi^2)_{\#}\g = \m_{s+t\varepsilon} = ((1-t)\pi^1 + t\pi^2)_{\#}\g_{s,\varepsilon}$ for all $t\in[0,1]$ we may express $\g_{s,\varepsilon}$ in terms of $\g$. In particular, the following equality holds.
\begin{equation}\label{pushforward}
\begin{split}
    &((1-t)\pi^1 + t\pi^2)_{\#}\bigg([(1-s)\p^1 + s\p^2,(1-(s+\varepsilon))\p^1 + (s+\varepsilon)\p^2]_{\#}\g\bigg)\\
    = & \bigg(((1-t)\pi^1 + t\pi^2)\circ [(1-s)\p^1 + s\p^2,(1-(s+\varepsilon))\p^1 + (s+\varepsilon)\p^2]\bigg)~_{\#}\g \\
    = & ((1-(s+t\varepsilon))\pi^1 + (s+t\varepsilon)\pi^2)_{\#}\g = ((1-t)\pi^1 + t\pi^2)_{\#}\g_{s,\varepsilon}.
\end{split}
\end{equation} 
From Equation \eqref{pushforward}, it follows that $\g_{s,\varepsilon} = [((1-s)\p^1 + s\p^2),((1-(s+\varepsilon))\p^1 + (s+\varepsilon)\p^2)]_{\#}\g$ and,   consequently, 
\begin{flalign*}
     \frac{d}{dt}_{+}F(\m_t)\bigg|_{t=s} & = \frac{1}{\varepsilon}\int_{(\domain)^2} \nabla_wF[\m_s](y_1)\cdot(y_2-y_1) \ d\g_{s,\varepsilon} (y_1,y_2)\\
     & = \int_{(\domain)^2} \nabla_wF[\m_s]((1-s)x_1+sx_2)\cdot(x_2-x_1) \ d\g (x_1,x_2).     
\end{flalign*}
For $s \in (0,1]$, it also follows from Lemma \ref{onesidedgeodesic} that there exists $\delta > 0$ such that $[s-\delta,s]\ni t\mapsto \m_t$ defines a geodesic. Consequently, (and using a similar reasoning to the calculation of the right-sided derivative) there exists an optimal plan $\g_{s,\delta} \in \G_o(\m_s,\m_{s-\delta})$ such that $((1-(s-t\delta))\pi^1 + (s-t\delta)\pi^2)_{\#}\g = \m_{s-t\delta} = ((1-t)\pi^1 + t\pi^2)_{\#}\g_{s,\delta}$ for all $t\in [0,1]$. Moreover, it follows from Lemma \ref{firstderivative} that
\[
 \frac{d}{dt}_{-}F(\m_t)\bigg|_{t=s} = -\frac{1}{\delta}\frac{d}{dt}_{-}F(\m_{s- t\delta})\bigg|_{t=0} = -\frac{1}{\delta}\int_{(\domain)^2} \nabla_wF[\m_s](y_1)\cdot(y_2-y_1) \ d\g_{s,\delta} (y_1,y_2). 
\]
As established in Equation \eqref{pushforward}, it follows that $\g_{s,\varepsilon} = [((1-s)\p^1 + s\p^2),((1-(s+\varepsilon))\p^1 + (s+\varepsilon)\p^2)]_{\#}\g$. By a similar reasoning, we deduce that $\g_{s,\delta} = [((1-s)\p^1 + s\p^2),((1-(s-\delta))\p^1 + (s-\delta)\p^2)]_{\#}\g$ and, consequently, 
\begin{flalign*}
     \frac{d}{dt}_{-}F(\m_t)\bigg|_{t=s} & = -\frac{1}{\delta}\int_{(\domain)^2} \nabla_wF[\m_s](y_1)\cdot(y_2-y_1) \ d\g_{s,\delta} (y_1,y_2)\\
     & = \int_{(\domain)^2} \nabla_wF[\m_s]((1-s)x_1+sx_2)\cdot(x_2-x_1) \ d\g (x_1,x_2).     
\end{flalign*}     
Since $s$ was an arbitrary point in the interval $(0,1)$, the left and right-sided derivatives of $t\mapsto F(\m_t)$ agree on $(0,1)$. Moreover, we conclude that $[0,1]\ni t\mapsto F(\m_t)$ is differentiable on $[0,1]$.
\end{proof}

\begin{lemma}\label{convexjoininglemma}
Let $f\co[a,c]\to\R$. If $f$ is differentiable on $[a,c]$ and there exists $b\in (a,c)$ such that the restriction of $f$ to $[a,b]$ and $[b,c]$ is convex, then $f$ is convex on $[a,c]$.
\end{lemma}
\begin{proof}
Since $f$ is differentiable, it follows that $f$ is convex on $[a,c]$ if and only $(f'(x)-f'(y))(x-y)\geq 0$ for all $a< x,y< c$ (cf. \cite[Theorem 12.18]{Binmore}). Without loss of generality, we fix $ a < x \leq y < c$. If $x,y \geq b$ or $x,y \leq b$ then, since $f$ is convex on $[a,b]$ and $[b,c]$, it follows as a consequence of \cite[Theorem 12.18]{Binmore} that $(f'(x)-f'(y))(x-y)\geq 0$. Subsequently, we consider the case $x\leq b \leq y$. Since $f$ is convex on $[a,b]$ and $[b,c]$, it again follows from Lemma \ref{monotonederivative} that $(f'(y)-f'(b))(y-b)\geq 0$ and $(f'(b)-f'(x))(b-x)\geq 0$. Moreover, since $x\leq b \leq y$, we know that $f'(y) \geq f'(b) \geq f'(x)$ and, consequently, $(f'(y)-f'(x))(y-x) \geq 0$. We conclude that $f$ must be convex on $[a,c]$.
\end{proof}

\begin{proof}[\hypertarget{proof:theoremone} Proof of Theorem \ref{geodesicequivalencetheorem}]
    The set of acceleration-free curves contains every generalised geodesic and the set of generalised geodesics contains every geodesic. Consequently, if $F$ is $\l$-convex along acceleration-free curves, it must also be $\l$-convex along generalised geodesics and, similarly, if $F$ is $\l$-convex along generalised geodesics then $F$ must also be $\l$-geodesically convex. It is left to show that $\l$-geodesic convexity implies $\l$-convexity along acceleration-free curves. We first consider the case in which $\l = 0$.
    
    Let $[0,1]\ni t\mapsto \m_t$ be an acceleration-free curve between $\m_1,\m_2 \in \sP_2^d(\domain)$. Given $s\in (0,1)$, Lemma \ref{onesidedgeodesic} implies that there exist $\delta,\varepsilon > 0$ such that the maps $[s-\delta,s]\ni t\mapsto \m_t$ and $[s,s+\varepsilon]\ni t\mapsto \m_t$ define geodesics. Moreover, since we assume $F$ to be geodesically convex, the restriction of $[0,1]\ni t\mapsto F(\m_t)$ to the intervals $[s-\delta,s]$ and $[s,s+\varepsilon]$ must be convex. By Lemma \ref{accelerationfreederivative}, the map $[0,1]\ni t\mapsto F(\m_t)$ is differentiable on $[0,1]$ and, consequently, we conclude from Lemma \ref{convexjoininglemma} that the restriction of the map $[0,1]\ni t\mapsto F(\m_t)$ to the interval $[s-\delta,s+\varepsilon]$ is also convex. Since the map $t\mapsto F(\m_t)$ is continuous on $[0,1]$ and convex on a neighbourhood of $s$ for every $s \in (0,1)$, we conclude that the map $[0,1]\ni t\mapsto F(\m_t)$ is convex. Now, since $t\mapsto \m_t$ was an arbitrary acceleration-free curve between measures in $\sP_2^d(\domain)$, we conclude that the restriction of $F$ to $\sP_2^d(\domain)$ must be convex along acceleration-free curves. Moreover, since $F$ is differentiable it must also be continuous, and so, as a consequence of Lemma \ref{empiricalapproximation}, $F$ must also be convex along all acceleration-free curves. 

    We now assume that $F\co\sP_2(\domain)\to\R$ is differentiable and $\l$-geodesically convex for $\l \neq 0$. By Remark \ref{lambdaconvexitycriterion}, the map $\sP_2(\domain)\ni\m\mapsto F(\m) - \frac{\l}{2}\int_{\domain}|x|^2 d\m (x)$ is differentiable and geodesically convex. As we have shown, this implies that the map $\sP_2(\domain)\ni\m\mapsto F(\m) - \frac{\l}{2}\int_{\domain}|x|^2 d\m (x)$ is convex along acceleration-free curves and so we conclude by Remark \ref{lambdaconvexitycriterion} that $F\co\sP_2(\domain)\to\R$ must be $\l$-convex along acceleration-free curves.
    \end{proof}
In the following example, we construct a function on $\sP_2(\R)$ which is both geodesically convex and continuous but not convex along acceleration-free curves. This demonstrates that, in general, one can not hope to relax the assumption of Theorem \ref{geodesicequivalencetheorem} from differentiability to continuity. 

\begin{example}\label{counterexample}
Let $\varepsilon> 0$. We define $\sW_\varepsilon\co \sP_2(\R)\to \R$, 
\[
\sW_\varepsilon(\m) := \int_{\R}\int_{\R} W_\varepsilon(x,y) \ d\m(x) \ d\m(y), \quad W_\varepsilon(x,y) := \left\{ \begin{aligned}
& \varepsilon - |x-y| \ \text{if} \ |x-y| \leq \varepsilon,\\
    & 0 \ \text{otherwise}.
\end{aligned} \right.
\]
We also recall the definition of the sets $\Delta^+$ and $\Delta^-$ from  \cite[Proposition 7.25]{Santambrogio2015}.
\[
    \Delta^+ := \{(x,y) \in \R^2, y \geq x\}, 
    \quad 
    \Delta^- := \{(x,y) \in \R^2, y \leq  x\}.
    \]
Since $W_\varepsilon$ is continuous and bounded, it follows from \cite[Proposition 7.2]{Santambrogio2015}, that $\sW_\varepsilon$ is continuous on $(\sP_2(\R),W_2)$. In addition, since the restriction of $W_\varepsilon$ to the sets $\Delta^+$ and $\Delta^-$ is convex, it follows from \cite[Proposition 7.25]{Santambrogio2015}, that $\sW_\varepsilon$ geodesically convex. However, as we now show, the function $\sW_\varepsilon$ is not convex along acceleration-free curves.
   Consider the measures $\m := \frac{1}{2}\delta_x + \frac{1}{2}\delta_{x'}$ and $\n := \frac{1}{2}\delta_{y}+ \frac{1}{2}\delta_{y'}$ and let $\g\in \G(\m,\n)$ be the plan such that $\supp(\g) = \{(x,y), (x',y')\}$. Furthermore, define $[0,1]\ni t\mapsto \m_t:= ((1-t)\pi^1 + t\pi^2)_{\#}\g$. 
  \begin{flalign*}
    \sW_\varepsilon(\m_t) &  = \frac{1}{2}W_\varepsilon((1-t)x+ty,(1-t)x'+ty') \\
      & + \frac{1}{4}W_\varepsilon((1-t)x+ty,(1-t)x+ty) + \frac{1}{4}W_\varepsilon((1-t)x'+ty',(1-t)x'+ty') \\
      & = \frac{1}{2}(W_\varepsilon((1-t)x+ty,(1-t)x'+ty')+\varepsilon).
  \end{flalign*}
  
  In particular, if we choose $(x,x')$ and $(y,y')$ such that $x-x' = C \geq \varepsilon$ and $y-y' = - C \leq -\varepsilon$ then $\sW_\varepsilon(\m) = \sW_\varepsilon(\n) = \frac{1}{2}\varepsilon$. However, it also follows that 
  \[
  \sW_\varepsilon(\m_\frac{1}{2}) = \frac{1}{2}(\varepsilon- \frac{1}{2}|(x-x')+(y-y')|+\varepsilon) = \frac{1}{2}(\varepsilon + \varepsilon) = \varepsilon.
  \]
  Since $\sW_\varepsilon(\m_\frac{1}{2}) > \frac{1}{2}(\sW_\varepsilon(\m)+ \sW_\varepsilon(\n))$, we conclude that $\sW_\varepsilon$ is not convex along acceleration-free curves.

  To emphasise that the function $\sW_\varepsilon$ is not differentiable, we pick $x,x',y,y'\in \R$ such that $x<x'$ and $y<y'$ and, subsequently, consider the geodesic $[0,1]\ni t\mapsto \m_t := \frac{1}{2}(\delta_{(1-t)x + ty} + \delta_{(1-t)x' + ty'})$.
  If $\sW_\varepsilon$ was differentiable on $\sP_2(\R)$ then, by Lemma \ref{firstderivative}, the map $[0,1]\ni t\mapsto \sW_\varepsilon(\m_t)$ would also be differentiable. However, if we choose $x' = x+\frac{\varepsilon}{2}$ and $y' = y + 
 2\varepsilon$, it follows that  
 \begin{flalign*}\sW_\varepsilon(\m_t) &  = \frac{1}{2}(W_\varepsilon((1-t)x+ty,(1-t)x'+ty')+\varepsilon)\\ & 
 = \frac{1}{2}(W_\varepsilon((1-t)x+ty,(1-t)x+ty +\frac{1+3t}{2}\varepsilon )+\varepsilon) = \left\{ \begin{aligned}
& \frac{3(1-t)}{4}\varepsilon \ \text{if} \ t \leq \frac{1}{3},\\
    & \frac{1}{2}\varepsilon \ \text{otherwise}.
\end{aligned} \right.
 \end{flalign*}
Moreover, since $[0,1]\ni t\mapsto \sW_\varepsilon(\m_t)$ is not differentiable, we conclude that $\sW_\varepsilon$ is not differentiable on $\sP_2(\R)$ either.
\end{example} 

%% file: chapters/chaptertwo.tex
\section{First and Second Order Convexity Criteria}
\label{sec:two}
In this section, we prove Theorems \ref{convexityone} and \ref{convexitytwo}. Whilst the proof of Theorem \ref{convexityone} is relatively self-contained, in contrast, and, as we outlined in \hyperref[subsec:two]{Subsection 1.4}, the proof of Theorem \ref{convexitytwo} requires us to first establish a number of helpful lemmas. 


\subsection*{First Order Convexity Criteria}

\begin{proof}[Proof of Theorem \ref{convexityone}]
     Firstly, we assume that $F$ is differentiable and $\l$-geodesically convex and show that Equation \eqref{displacementmonotone} holds. 
    If $F$ is differentiable and $\l$-geodesically convex then, for any geodesic $[0,1]\ni t\mapsto \m_t$ with endpoints $\m_1,\m_2\in \sP_2(\domain)$, the map $[0,1]\ni t\mapsto F(\mu_t)$ is differentiable by Lemma \ref{firstderivative} and $\l W_2^2(\m_1,\m_2)$-convex by Remark \ref{everygeodesicremark}. 
    Using the convexity characterisation of Lemma \ref{monotonederivative}, the following inequality holds. 
    \begin{equation}\label{onesidedinequality}
    \frac{d}{dt}_{-}F(\m_t) \bigg|_{t=1}-\frac{d}{dt}_{+}F(\m_t) \bigg|_{t=0}\geq\l W_{2}^2(\m_1,\m_2). 
    \end{equation}
    Furthermore, using Lemma \ref{firstderivative} to characterise the one-sided derivatives of $t\mapsto F(\m_t)$, we conclude that Equation \eqref{onesidedinequality} and Equation \eqref{displacementmonotone} are equivalent and, since the choice of geodesic was arbitrary, we conclude that \eqref{displacementmonotone} must hold for all $\m_1,\m_2$ and all $\g \in \G_o(\m_1,\m_2)$.
    
    We now assume that $F$ is differentiable and that Equation \eqref{displacementmonotone} holds for all $\m_1,\m_2$ and all $\g \in \G_o(\m_1,\m_2)$ and show that $F$ is also $\l$-geodesically convex. 
    Let $[0,1]\ni t\mapsto \m_t$ be a geodesic with endpoints $\m_1,\m_2\in \sP_2(\domain)$. By Lemma \ref{geodesicremark}, the curve $[0,1]\ni t\mapsto \mu_t^{r,s}:= \mu\circ ((r-s)t+s)$ must be the unique geodesic between $\m_s$ and $\m_r$. 
    Moreover, since Equation \eqref{onesidedinequality} and Equation \eqref{displacementmonotone} are equivalent, the following inequality must hold for all $0 \leq s<r \leq 1$.
    \begin{equation}\label{rsinequality}
    \frac{d}{dt}_{-}F(\mu_t^{r,s})\bigg|_{t=1} -\frac{d}{dt}_{+}F(\mu_t^{r,s})\bigg|_{t=0} 
    \geq \l W_{2}^2(\m_r,\m_s).      
    \end{equation}
    Subsequently, the following inequality must hold for any $0\leq s<r\leq 1$.
    \begin{flalign*}
    \frac{d}{dt}F(\mu_t)\bigg|_{t=r} -\frac{d}{dt}F(\mu_t)\bigg|_{t=s} & = 
    \frac{1}{r-s}\bigg(\frac{d}{dt}_{-}F(\mu_t^{r,s})\bigg|_{t=1} -\frac{d}{dt}_{+}F(\mu_t^{r,s})\bigg|_{t=0}\bigg) \\
    & \geq \l \frac{r-s}{(r-s)^2}W_{2}^2(\m_r,\m_s) = \l W_2^2(\m_1,\m_2)(r-s).          
    \end{flalign*}
    Since the choice of geodesic was arbitrary, we conclude that there exists a geodesic $[0,1]\ni t\mapsto \m_t$ between any pair of measures $\m_1,\m_2 \in \sP_2(\domain)$ such that the function $[0,1]\ni t\mapsto F(\mu_t)$ is $\l W_2^2(\m_1,\m_2)$ convex. 
    Subsequently, we conclude that $F$ is $\l$-geodesically convex.
\end{proof}

%% file: chapters/chapterthree.tex
\subsection*{Second Order Convexity Criteria.}
To begin this subsection, we first calculate the second derivative of the map $t\mapsto F(\m_t)$ when $F$ is twice differentiable and $[0,1]\ni t\mapsto \m_t$ defines a suitably `smooth' geodesic.

\begin{lemma}\label{secondderivative}
Let $F\co\sP_2(\domain)\to\R$ be twice differentiable on $\sP_2(\domain)$ and let $\m_1,\m_2\in \sP_2(\domain)$ be measures for which there exists an optimal map $T$ between $\m_1$ and $\m_2$. 
Let $[0,1]\ni t \mapsto \m_t:= ((1-t)id +tT)_{\#}\m_1$ and let $T_t$ denote the optimal map between $\m_t$ and $\m_2$.
If there exists $\varphi\in C_c^\infty(\domain)$ and $U$, a convex neighbourhood of $\Hull(\supp(\m_1))$, such that the restriction of $\varphi$ to $U$ is $(-1)$-convex and such that $T-id = \nabla \varphi$, then the map $[0,1]\ni t\mapsto F( \m_t)$ is twice differentiable on $[0,1)$. 
In particular,
\[
\frac{d^2}{dt^2} F(\m_t)  = \frac{1}{(1-t)^2}Hess F[ \m_t ](T_t-id,T_t-id).
\]
\end{lemma}

\begin{proof}
     For $t\in [0,1)$, we define $g_t\co U\to \domain, \ g_t(x) := \nabla(\frac{1}{2}|x|^2 + t\varphi(x))$. 
     Since $g_t$ is the gradient of a smooth strictly convex function, $g_t$ is invertible and its inverse $g_t^{-1}\co g_t(U) \to U$ is smooth. 
     In particular, using the Inverse Function Theorem, the following system of equalities holds.
    \begin{equation}\label{nablaccinfinity}
    \begin{split}
    & \nabla (\frac{t}{2} |\nabla\varphi\circ g_t^{-1}|^2 + \varphi\circ g_t^{-1}) = \nabla (g_t^{-1})(t (\nabla^2\varphi\circ g_t^{-1})\nabla\varphi\circ g_t^{-1} + \nabla\varphi\circ g_t^{-1}) \\
    & = (\nabla g_t)^{-1}((t\nabla^2\varphi\circ g_t^{-1} + I_{d})\nabla\varphi\circ g_t^{-1}) =  (\nabla (t\nabla\varphi+id))^{-1}((t\nabla^2\varphi\circ g_t^{-1} + I_{d})\nabla\varphi\circ g_t^{-1})  \\ 
   & = (t\nabla^2\varphi\circ g_t^{-1} + I_{d})^{-1}(t\nabla^2\varphi\circ g_t^{-1} + I_{d})(\nabla\varphi\circ g_t^{-1}) = \nabla\varphi\circ g_t^{-1}.    
    \end{split}
    \end{equation}
    From Equation \eqref{nablaccinfinity}, it follows that $\nabla\varphi \circ \ g_t^{-1} \in \nabla C^\infty_c(g_t(U))$.
    Moreover, there exists $f_t\in C_c^\infty(\domain)$ such that $\nabla f_t|_{g_t(U)} = \nabla\varphi \circ g_t^{-1}$. 
    It also follows from Lemma \ref{geodesicremark} that there is a unique optimal plan between $\m_t$ and $\m_s$ for all $s\in[0,1]$ and $t\in[0,1)$. 
    In the following system of equalities, we show that this optimal plan is induced by an optimal map of the form $id + (s-t)(\nabla\varphi\circ g_t^{-1})$.
    \begin{flalign*}
    \int_{\domain} |id - & (id + (s-t)(\nabla\varphi\circ g_t^{-1}))|^2 \ d\m_t =  (s-t)^2 \int_{\domain} |\nabla\varphi \circ g_t^{-1}|^2 \  d\m_t \\
    = & (s-t)^2  \int_{\domain} |(T - id) \circ [(id + t(T-id))^{-1}]|^2 \ d\m_t \\
     = & (s-t)^2  \int_{\domain} |(T - id) \circ [(id + t(T-id))^{-1}]|^2 \ d(id + t(T-id))_{\#}\m_1 \\
    = & (s-t)^2  \int_{\domain} |T-id|^2 \ d\m_1 = (s-t)^2 W_2^2(\m_1,\m_2) = W_2^2(\m_t,\m_s) .
    \end{flalign*}
    Now, since $U$ is a neighbourhood of $\Hull(\supp(\m_1))$ and $\m_t = (g_t)_{\#}\m_1$, it follows that $\supp(\m_t)$ is a subset of $g_t(U)$. Moreover, since the optimal map $id + (s-t)(\nabla\varphi\circ g_t^{-1})$ is only uniquely defined $\m_t$-almost everywhere and since $\nabla f_t|_{g_t(U)} = \nabla\varphi \circ g_t^{-1}$, we may identify the optimal map with $id+(s-t)\nabla f_t$. In particular, this means that $T_t = id+(1-t)\nabla f_t$.

    Utilising Definition \ref{differentiability}, we calculate the first derivative of $F\circ\m_t$ as follows.
    \begin{flalign*}
    \frac{d}{dt} F(\m_t) & = \lim_{h\to0}\frac{1}{h}(F(\m_{t+h}) - F(\m_{t}))  = \lim_{h\to0}\frac{h}{h}\int_{\domain}\nabla_wF[\m_t]\cdot \nabla f_t  \ d\m_t + \lim_{h\to0}\frac{o(W_2(\m_{t+h},\m_{t}))}{h}\\
    & 
    =  \int_{\domain}\nabla_wF[\m_t]\cdot \nabla f_t \ d\m_t  = \nabla f_t \cdot F[\m_t] =  \frac{1}{1-t}(T_t-id) \cdot F[\m_t].
   \end{flalign*}
   Since $F$ is twice differentiable, it follows from Definition \ref{Hessiandefn} that the map $t \mapsto \z \cdot F[\m_t]$ is differentiable for all $\z\in \nabla C_c^\infty(\domain)$. 
   In particular, for $s,r\in[0,1)$,
    \[
   \frac{d}{dt} \nabla f_s \cdot F[\m_t]\bigg|_{t=r} = \nabla f_r \cdot(\nabla f_s \cdot F[\m_r]).
   \]
   We calculate the second derivative as follows.
   \begin{flalign*}
    \frac{d^2}{dt^2}F(\m_t) & = \lim_{h\to0}\frac{1}{h}(\nabla f_{t+h}\cdot F[\m_{t+h}] -  \nabla f_{t}\cdot F[\m_t]) \\
    & =  \lim_{h\to0}\frac{1}{h}(\nabla f_{t+h}\cdot F[\m_{t+h}] - \nabla f_{t+h} \cdot F[\m_t]) +\lim_{h\to0}\frac{1}{h}(\nabla f_{t+h} \cdot F[\m_t] - \nabla f_t\cdot F[\m_t]) \\
    & = \lim_{s\to t}\frac{d}{dt} (\nabla f_s\cdot \nabla_w F[\m_t]) +  (\partial_t \nabla f_t) \cdot F[\m_t] = \nabla f_t\cdot (\nabla f_t\cdot \nabla_w F[\m_t]) +  (\partial_t \nabla f_t) \cdot F[\m_t]. 
    \end{flalign*}
    Since $\nabla f_t|_{g_t(U)} = \nabla\varphi \circ g_t^{-1}$ and via the application of the chain rule, the following system of equalities holds $\m_t$-almost everywhere.
    \begin{equation}\label{equation1}
    \begin{split}
    & \partial_t \nabla f_t + \nabla^2 f_t \nabla f_t = \partial_t(\nabla\varphi\circ g_t^{-1}) +  \nabla(\nabla\varphi\circ g_t^{-1})\nabla\varphi\circ g_t^{-1} \\
    & = (\nabla^2 \varphi\circ g_t^{-1})(\partial_t (g_t^{-1}) + \nabla (g_t^{-1})(\nabla\varphi\circ g_t^{-1}))
     = (\nabla^2 \varphi\circ g_t^{-1})(\partial_t (g_t^{-1}) + \nabla (g_t^{-1})(\partial_t g_t \circ g_t^{-1})).
    \end{split}
    \end{equation}
    Via the application of the chain rule, we also derive Equation \eqref{equation2}.
    \begin{equation}\label{equation2}
    \begin{split}
    0 & = 0 \circ g_t^{-1}  = \partial_t id \circ g_t^{-1} = \partial_t (g_t^{-1}\circ g_t)\circ g_t^{-1}\\
    & =  (\partial_t g_t^{-1} )\circ g_t \circ g_t^{-1} + (\nabla (g_t^{-1})\circ g_t\circ g_t^{-1})(\partial_t g_t) \circ g_t^{-1}=  \partial_t(g_t^{-1}) + \nabla g_t^{-1}(\partial_t g_t \circ g_t^{-1})
    \end{split}
    \end{equation}
    By substituting Equation \eqref{equation2} into Equation \eqref{equation1}, it follows that $\partial_t \nabla f_t + \nabla^2 f_t \nabla f_t = 0$, $\m_t$-almost everywhere.
    We conclude as follows.
    \begin{flalign*}
    \frac{d^2}{dt^2} F(\m_t) & =   \nabla f_t\cdot (\nabla f_t\cdot \nabla_w F[\m_t]) +  (\partial_t\nabla f_t)\cdot F[\m_t] 
    = \nabla f_t\cdot (\nabla f_t\cdot \nabla_w F[\m_t]) -  \nabla^2 f_t \nabla f_t  \cdot F[\m_t]\\
    & = \bar{H}essF[ \m_t ]( \nabla f_t, \nabla f_t) = \frac{1}{(1-t)^2}Hess F[ \m_t ](T_t-id,T_t-id).
    \end{flalign*}
\end{proof}

\begin{lemma}\label{approximationlemma}
Let $\m,\n\in\sP_2^{rc}(\domain)$ and let $T$ denote the optimal map between $\m$ and $\n$. 
There exists a sequence of $C_c^\infty(\domain)$ functions $(\varphi_n)_{n\in\N}$ and $U_n$, a convex neighbourhood of $\Hull(\supp(\m))$, such that each $\varphi_n$ is $(-1)$-convex on $U_n$ and $(\nabla \varphi_n + id) \to T$ in $L^2(\m;\domain)$ as $n\to\infty$.
\end{lemma}
\begin{proof}
In \cite{Ambrosio2008}, Theorem 6.2.10, it is shown that, for $\m \in \sP_2^r(\domain)$, there exists a convex function $\phi$ such that $\nabla \phi = T$ in $L^2(\m;\domain)$. 
It is also shown that, if $\n$ is compactly supported, then $\phi$ is locally Lipschitz. 
Since $\n$ is compactly supported, it also follows that $T$ is bounded $\m$-almost everywhere.
Let $\chi_\varepsilon$ be a positive mollifier and let $(\k_n)_{n\in\N}$ be a sequence of smooth cutoff functions such that $\k_n = 1$ and $\nabla\k_n = 0 $ on $B_{n}(0)$. 
Since $\m$ has compact support, there exists $N \in \N$ such that $\Hull(\supp(\m))\subset B_{n}(0)$ for all $n\geq N$. 
We fix such an $N$, we define $\varphi_n := \k_{n+N}(\chi_\frac{1}{n}\ast \phi - \frac{1}{2}|id|^2)$ and we let $U_n = B_{n+N}(0)$. 
Since mollification preserves convexity and $\k_{n+N}(x) =1$ for $x\in U_n$, each $\varphi_n$ is $(-1)$-convex on $U_n$ for all $n\in \N$. 
Now, using the convergence properties of a mollifier, the gradient $\nabla\varphi_n(x)$ converges to $T(x) - x$ for every $x\in \Hull(\supp(\m))$. 
Moreover, since $\m$ is absolutely continuous with respect to $\mathscr{L}^d$, it follows that $\nabla\varphi_n(x)$ converges to $T(x)-x$ for $\m$-almost every $x\in \domain$. 
Using this convergence and the boundedness of $T$, it follows that there exists $C\in \R$ such that $|\nabla\varphi_n| \leq C$ $\m$-almost everywhere. 
Moreover, by applying the Dominated Convergence Theorem to the sequence $\nabla\varphi_n$, it follows that $(\nabla \varphi_n + id)$ converges to $T$ in $L^2(\m;\domain)$ as $n\to\infty$. 
\end{proof}
Whilst the following results, Lemmas \ref{convergencelemma} and \ref{geodesicconvexityspace} are not used directly in the Proof of Theorem \ref{convexitytwo}, they are necessary results in the proof of Lemma \ref{convexitycriterion} which, in turn, is utilised in our proof of the main result. Moreover, whilst it is expected that the following two lemmas are well-known to experts, we include their proof for completeness.

\begin{lemma}\label{convergencelemma}
Let $(\m_{1,n})_{n\in\N}$ and $(\m_{2,n})_{n\in\N}$ be sequences in $\sP_2(\domain)$ and let $(\g_n)_{n\in\N}$ be a sequence such that $\g_n\in \G_o(\m_{1,n},\m_{2,n})$ for all $n \in \N$. 
If there exist $\m_1,\m_2\in \sP_2(\domain)$ such that $\lim_{n\to\infty}W_2(\m_1,\m_{1,n})= 0$ and $\lim_{n\to\infty}W_2(\m_2,\m_{2,n}) = 0$ then there exists $\g \in \G_o(\m_1,\m_2)$ such that, up to a subsequence, $(\g_n)_{n\in\N}$ narrowly converges to $\g$.
\end{lemma}

\begin{proof}
It is shown in \cite{Ambrosio2008}, Proposition 7.1.3, that, if $W_2^2(\m_{1,n},\m_{2,n})$ is bounded and $(\m_{1,n})_{n\in\N}$ and $(\m_{2,n})_{n\in\N}$ narrowly converge to $\m_1$ and $\m_2$ respectively, then $(\g_n)_{n\in\N}$ is relatively compact with respect to the narrow convergence and any narrow limit point is an element of $\G_o(\m_1,\m_2)$. 
It is also shown in \cite{Ambrosio2008}, Proposition 7.1.5, that a sequence of probability measures $(\m_n)_{n\in\N}$ satisfies $\lim_{n\to\infty} W_2(\m,\m_n) = 0$ if and only if $(\m_n)_{n\in\N}$ narrowly converges to $\m$ and the sequence $\int_{\domain} |x|^2 \m_n(x)$ is uniformly bounded. 
Consequently, if we assume that there exist $\m_1,\m_2\in \sP_2(\domain)$ such that $\lim_{n\to\infty}W_2(\m_1,\m_{1,n})=0$ and $\lim_{n\to\infty}W_2(\m_2,\m_{2,n})=0$ then, by \cite[Proposition 7.1.5]{Ambrosio2008}, it follows that $(\m_{1,n})_{n\in\N}$ and $(\m_{2,n})_{n\in\N}$ narrowly converge to $\m_1$ and $\m_2$ respectively. 
Moreover, by \cite[Proposition 7.1.3]{Ambrosio2008}, it follows that the sequence $(\g_n)_{n\in\N}$ is relatively compact with respect to the narrow convergence and any narrow limit point is an element of $\G_o(\m_1,\m_2)$. 
Consequently, there exists $\g \in \G_o(\m_1,\m_2)$ such that, up to a subsequence, $(\g_n)_{n\in\N}$ narrowly converges to $\g$.
\end{proof}

\begin{propn}\label{geodesicconvexityspace}
    The space $\sP_2^{rc}(\domain)$ is geodesically convex.
\end{propn}
\begin{proof}
    In \cite{McCann1997}, Proposition 1.3, it is stated that, for a geodesic $[0,1]\ni t \mapsto \m_t$ with endpoints $\m_1,\m_2 \in \sP_2^{r}(\domain)$, the measure $\m_t \in \sP^r_2(\domain)$ for all $t\in[0,1]$. 
    Consequently, it is left to show that geodesic interpolations preserve compactness of the support. Let $\m_1,\m_2\in \sP_2^{rc}(\domain)$ and let $\g\in \G_o(\m_1,\m_2)$ so that the curve $[0,1]\ni t\mapsto \m_t:= ((1-t)\pi^1+t\pi^2)_{\#}\g$ defines a geodesic between $\m_1$ and $\m_2$. 
    Given any Borel set $A\subset\domain$ and $t\in[0,1]$, it follows from the definition of the pushforward that there exists a set $B\subset \domain \times \domain$ such that $\m_t(A)= \g(B)$. 
    Moreover, since the first and second marginals of $\g$ have compact support, it follows that $\g$ must be compactly supported and, consequently, $\m_t$ must also be compactly supported.
\end{proof}

\begin{lemma}\label{convexitycriterion}
If $F\co\sP_2(\domain)\to\R$ is continuous then $F$ is $\l$-geodesically convex if and only if the restriction of $F$ to $\sP_2^{rc}(\domain)$ is $\l$-geodesically convex.
\end{lemma}

\begin{proof}
        By Proposition \ref{geodesicconvexityspace}, the set $\sP_2^{rc}(\domain)$ is geodesically convex. 
        Consequently, if $F$ is $\l$-geodesically convex, then the restriction of $F$ to $\sP_2^{rc}(\domain)$ is $\l$-geodesically convex and it is left to prove the converse implication. 
       Let $\m_1,\m_2 \in \sP_2(\domain)$. 
       Since $\sP_2^{rc}(\domain)$ is dense in $\sP_2(\domain)$, there exists at least a pair of sequences in $\sP_2^{rc}(\domain)$ which converge in $(\sP_2(\domain),W_2)$ to $\m_1$ and $\m_2$ respectively. 
       We fix one such pair, denoting the sequences $(\m_{1,n})_{n\in\N}$ and $(\m_{2,n})_{n\in\N}$. 
       Additionally, we let $(\g_n)_{n\in\N}$ be the unique sequence of optimal plans satisfying $\g_n \in \G_o(\m_{1,n},\m_{2,n})$ for all $n\in \N$. Using Lemma \ref{convergencelemma}, we extract a subsequence of $(\g_n)_{n\in\N}$ which narrowly converges to $\g\in \G_o(\m_1,\m_2)$. 
       We denote this subsequence $(\g_n)_{n\in\N}$ and denote its marginals $(\m_{1,n})_{n\in\N}$ and $(\m_{2,n})_{n\in\N}$ to avoid relabeling. 
       Now, for $t\in[0,1]$, let $\m_t^n := ((1-t)\pi^1 + t\pi^2)_{\#}\g_n$ and let $\m_t := ((1-t)\pi^1 + t\pi^2)_{\#}\g$. 
       Since $\g_n$ converges narrowly to $\g$, it follows that $\m_t^n$ converges narrowly to $\m_t$ for all $t\in [0,1]$. 
       Moreover, since the sequence $\int_{\domain}|x|^2\m_t^n(x)$ is uniformly bounded, it follows from \cite[Proposition 7.1.5]{Ambrosio2008} that $\lim_{n\to\infty} W_2(\m_t^n,\m_t)= 0$ for all $t\in [0,1]$. 
       If we now assume that $F\co\sP_2(\domain)\to\R$ is continuous and its restriction to $\sP_2^{rc}(\domain)$ is $\l$-geodesically convex, then the following system of inequalities holds for all $t\in[0,1]$.
       \begin{flalign*}
       F(\m_t) = \lim_{n\to\infty} F(\m_t^n) & \leq (1-t)\lim_{n\to\infty} F(\m_{1,n})  + t\lim_{n\to\infty} F(\m_{2,n}) - t(1-t)\frac{\l}{2}\lim_{n\to\infty}  W_2^2(\m_{1,n},\m_{2,n}) \\
       &= (1-t) F(\m_{1}) + tF(\m_2) - t(1-t)\frac{\l}{2} W_2^2(\m_1,\m_2).
       \end{flalign*}
       Since $\m_1$ and $\m_2$ were chosen arbitrarily, we conclude that $F$ is $\l$-geodesically convex.
\end{proof}

\begin{proof}[\hypertarget{proof:theoremtwo} Proof of Theorem \ref{convexitytwo}]
    We first assume that $F$ is $\l$-geodesically convex and prove that Equation \eqref{displacementmonotonetwo} holds for all $\m \in \sP_2(\domain)$ and all $\z \in T_\m\sP_2(\domain)$. 
    Choose $\m \in \sP_2(\domain)$ and $\varphi \in C_c^\infty(\domain)$. Since $\varphi$ has bounded derivatives of every order, there exists $c \in \R$ such that $\frac{1}{c}\varphi + \frac{1}{2}|x|^2$ is convex. 
    It is shown in \cite{Santambrogio2015}, Theorem 1.48, that, if $\xi$ is the gradient of a convex, differentiable function and $\x \in L^2(\m;\domain)$, then $\x$ defines an optimal map between $\m$ and $\x_{\#}\m$. 
    Consequently, $\nabla(\frac{1}{c}\varphi + \frac{1}{2}|x|^2)$ defines an optimal map between $\m$ and $(\frac{1}{c}\nabla\varphi + id)_{\#}\m$. We let $[0,1]\ni t \mapsto \m_t\in \sP_2(\domain)$ denote the geodesic induced by this transport. 
    By Lemma \ref{secondderivative}, the one-sided second derivative of $[0,1]\ni t\mapsto F(\m_t)$ exists at $t=0$ and the following equality holds.
    \[
    c^2\frac{d^2}{dt^2}_{+}F(\m_t) \bigg|_{t=0} =  HessF[\m](\nabla\varphi,\nabla\varphi).
    \]
    Furthermore, since we have assumed that $F$ is $\l$-geodesically convex, the second derivative of $t\mapsto F(\m_t)$ is bounded below by $\l W_2^2(\m,(\frac{1}{c}\nabla\varphi + id)_{\#}\m)$. Consequently, the following inequality holds.
    \[
    HessF[\m](\nabla\varphi,\nabla\varphi) = c^2\frac{d^2}{dt^2}_{+}F(\m_t) \bigg|_{t=0} \geq c^2\l W_2^2(\m,(\frac{1}{c}\nabla\varphi + id)_{\#}\m) =  \l\| \nabla\varphi \|^2_{L^2(\m)}.
    \]
    As $\m$ and $\varphi$ were arbitrary, Equation \eqref{displacementmonotonetwo} must hold for all $\m \in \sP_2(\domain)$ and $\z \in \nabla C_c^\infty(\domain)$. 
    Moreover, by Definition \ref{Hessiandefn}, $HessF[\m]$ is continuous and, consequently, Equation \eqref{displacementmonotonetwo} must also hold for any $\m \in \sP_2(\domain)$ and $\z \in T_\m\sP_2(\domain)$ since $\nabla C_c^\infty(\domain)$ is dense in  $T_\m\sP_2(\domain)$.
    
    We now suppose that Equation \eqref{displacementmonotonetwo} holds for all $\m\in\sP_2(\domain)$ and $\z \in T_\m\sP_2(\domain)$ and prove that $F$ is $\l$-geodesically convex. 
    We let $\m,\n \in \sP_{2}^{rc}(\domain)$, we let $T$ be the optimal map between $\m$ and $\n$ and we let $[0,1]\ni t\mapsto \m_t$ be the geodesic induced by this transport.
    By Lemma \ref{approximationlemma}, there exists a sequence of $C_c^\infty(\domain)$ functions $(\varphi_n)_{n\in\N}$ and $U_n$, a convex neighbourhood of $\Hull(\supp(\m))$, such that each $\varphi_n$ is $(-1)$-convex on $U_n$ and $\nabla \varphi_n\to (T-id)$ in $L^2(\m;\domain)$ as $n\to\infty$. 
    Since $\nabla\varphi_n + id$ is the gradient of a function whose restriction to $U_n$ is convex, it follows, again from \cite[Theorem 1.48]{Santambrogio2015}, that $\nabla\varphi_n + id$ must be the unique optimal map between $\m$ and $(\nabla\varphi_n + id)_{\#}\m$.
    We let $T_n := \nabla\varphi_n + id $, we let $[0,1]\ni t\mapsto \m_t^n :=((1-t)id + tT_n)_{\#}\m$ and we let $T_{t,n}$ denote the unique optimal map between $\m_t^n$ and $({T_n})_{\#}\m$. 
    By Lemma \ref{secondderivative}, the map $[0,1]\ni t\mapsto F(\m^n_t)$ is twice differentiable on $[0,1)$ and, in particular, the following system of inequalities holds for all $t \in [0,1)$.
    \begin{flalign*}
    \frac{d^2}{dt^2}  F(\m^n_t) & = \frac{1}{(1-t)^2}HessF[\m_t^n](T_{t,n}-id,T_{t,n}-id) \geq  \frac{1}{(1-t)^2}\l  \| T_{t,n}-id \| _{L^2(\m_t^n)}^2\\
    & =   \frac{1}{(1-t)^2}\l W_2^2(\m_t^n,{T_n}_{\#}\m) = \l W_2^2(\m,{T_n}_{\#}\m) .
    \end{flalign*} 
The above inequality implies that $[0,1]\ni t\mapsto F(\m_t^n)$ is $ \l  W_2^2(\m,{T_n}_{\#}\m)$-convex. 
Furthermore, as $T_n$ converges to $T$ in $L^2(\m;\domain)$, it follows that $\lim_{n\to\infty} W_2(\m_t^n,\m_t) = 0$ for all $t\in [0,1]$. 
Since $F$ is continuous, this means that $[0,1]\ni t \mapsto F(\m_t)$ is $\l W_2^2(\m,\n)$-convex.
Moreover, since $\m$ and $\n$ were arbitrary measures in $\sP_2^{rc}(\domain)$, the restriction of $F$ to $\sP_2^{rc}(\domain)$ is $\l$-geodesically convex. We subsequently conclude,  by Lemma \ref{convexitycriterion}, that $F$ must be $\l$-geodesically convex on $\sP_2(\domain)$.
\end{proof}

%% file: chapters/acknowledgements.tex
\subsection*{Acknowledgements:}
The author would like to thank A. R. M\'esz\'aros for his valuable advice and stimulating discussions, for many useful comments, and for his direction toward numerous references used in this manuscript. The author also wishes to thank F. Santambrogio and W. Gangbo for their helpful comments and constructive feedback on an earlier version of this manuscript and extends their thanks to G. E. Sodini for pointing out the connections between this manuscript and \cite{cavagnari}. This work was supported by the Engineering and Physical Sciences Research Council [Grant Number EP/W524426/1].